\documentclass[11pt]{amsart}
\usepackage[utf8]{inputenc}
\usepackage{amssymb, enumerate}
\newtheorem{theorem}{Theorem}[section]
\newtheorem{lemma}[theorem]{Lemma}
\newtheorem{proposition}[theorem]{Proposition}
\newtheorem{corollary}[theorem]{Corollary}
\theoremstyle{definition}
\newtheorem{definition}[theorem]{Definition}
\newtheorem{example}[theorem]{Example}
\newtheorem{question}[theorem]{Question}
\theoremstyle{remark}
\newtheorem{remark}[theorem]{Remark}
\numberwithin{equation}{section}
\DeclareMathOperator{\diam}{diam}
\DeclareMathOperator{\dist}{dist}
\DeclareMathOperator{\esssup}{ess\,sup}

\title[Removable sets for intrinsic metric]{Removable sets for intrinsic metric and for holomorphic functions}

\author{Sergei Kalmykov}
\address{School of mathematical sciences, Shanghai Jiao Tong University, 800 Dongchuan RD, Shanghai 200240, China} 
\email{sergeykalmykov@inbox.ru}
\thanks{First author supported by NSFC grant 11650110426.}

\author{Leonid V. Kovalev}
\address{215 Carnegie, Mathematics Department, Syracuse University, Syracuse, NY 13244, USA}
\email{lvkovale@syr.edu}
\thanks{Second author supported by the National Science Foundation grant DMS-1362453.}

\author{Tapio Rajala}
\address{Department of Mathematics and Statistics, University of Jyvaskyla, P.O. Box 35
(MaD), FI-40014 University of Jyvaskyla, Finland}
\email{tapio.m.rajala@jyu.fi}
\thanks{Third author supported by the Academy of Finland project no. 274372}

\subjclass[2010]{Primary 28A78; Secondary 26A16, 30C62, 30H05, 49Q15, 51F99}
\keywords{Intrinsic metric, removable set, holomorphic function, quasiconvex domain, Hausdorff measure}

\begin{document}

\begin{abstract} We study the subsets of metric spaces that are negligible for the infimal length of connecting curves; such sets are called metrically removable. In particular, we show that every totally disconnected set with finite Hausdorff measure of codimension 1 is metrically removable, which answers a question raised by Hakobyan and Herron. The metrically removable sets are shown to be related to other classes of ``thin'' sets that appeared in the literature. They are also related to the removability problems for classes of holomorphic functions with restrictions on the derivative. 
\end{abstract}

\maketitle
\baselineskip6mm

\section{Introduction}

The studies of removable sets have a long history in complex analysis and geometric function theory~\cite{Younsi}. Removability may be defined in terms of either a function class (e.g., bounded holomorphic functions) or of some geometric quantity (e.g., extremal distance as in~\cite{AhlforsBeurling}). Our starting point is a purely geometric concept of removability, which makes sense in an abstract metric space. 

\begin{definition}\label{defmetricremovable} 
Let $(X, d)$ be a metric space. A set $E\subset X$ is \emph{metrically removable} if for any $\epsilon>0$, any two points $a,b\in X$ can be connected by a curve that is 
disjoint from $E\setminus\{a,b\}$ and has length at most $d(a,b)+\epsilon$. \end{definition}

Thus, the complement of a metrically removable set is $C$-quasiconvex for every $C>1$ (see Definition~\ref{defquasiconvex}).  Hakobyan and Herron~\cite{HakobyanHerron} posed the following question:  

\begin{question}\label{questionHH} Suppose $E\subset\mathbb R^n$ is a totally disconnected compact set with $\mathcal H^{n-1}(E)<\infty$. Does it follow that its complement is quasiconvex?
\end{question}

Question~\ref{questionHH} turns out to be equivalent to asking whether $E$ is metrically removable (Proposition~\ref{quasiconvex2removable}). We answer it affirmatively:

\begin{theorem}\label{metricallyremovable} 
If $E\subset \mathbb R^n$ is closed, totally disconnected, and $\mathcal H^{n-1}(E)<\infty$, then $E$ is metrically removable.  
\end{theorem}

Quantitative control on the length and shape of connecting curves is important for recovering the properties of a holomorphic function $f$ from its derivative $f'$. This is the subject of sections~\ref{secBoundedDer} and~\ref{secargremovable}, which concern the removability of sets for holomorphic functions with restrictions on either the modulus or the argument of $f'$. This line of investigation involves the comparison  of different thinness conditions in~\S\ref{secThinness}, such as \emph{intervally thin sets} introduced by Tabor and Tabor~\cite{TaborTabor} in the context of convex analysis. Along the way we prove an extension theorem for $\delta$-monotone maps (Theorem~\ref{extendDM}) which is of independent interest. The paper concludes with remarks and questions in section~\ref{secQuestions}. 

\section{Notation and definitions}\label{secNotation}

 For $a,b\in\mathbb R^n$, $|a|$ is the Euclidean norm, $\langle a,b\rangle$ is the inner product, and $[a,b]$ is the line segment $\{(1-t)a+tb\colon 0\le t\le 1\}$. We write $B(a,r)$ for the open ball of radius $r$ with center $a$, and $\overline{B}(a,r)$ for the corresponding closed ball. The complement of a set $E$ is denoted $E^c$.

A \emph{curve} in a metric space $X$ is a continuous map $\gamma\colon [\alpha, \beta]\to X$. Its length $\ell(\gamma)$ is the supremum of the sums $\sum |\gamma(t_j)-\gamma(t_{j-1})|$ over all finite partitions $\{t_j\}$ of the interval $[\alpha,\beta]$. We also write $\gamma$ for $\gamma([\alpha, \beta])$ when parameterization is not important. 

\begin{definition}\label{defIntrinsic} 
The \emph{intrinsic metric} on a set $A\subset X$, written $\rho_A(a,b)$, is the infimum of the length of curves that connect $a$ to $b$ within $A$. This is indeed a metric when $A$ is connected by rectifiable curves; otherwise $\rho_A$ may take on the value $\infty$ although the other axioms of a metric still hold.  
\end{definition}

When a set $E$ is metrically removable, $\rho_{E^c}(a,b) = d(a,b)$ for all $a,b\in E^c$. The converse is also true when $E$ has empty interior; see Proposition~\ref{equivalentdef}. The property $\rho_{E^c}(a,b)=d(a,b)$ can  be expressed by saying that $E^c$ is a \emph{length space}~\cite[p.~28]{BuragoBuragoIvanov}. It is also related to the concept of quasiconvexity. 
 
\begin{definition}\label{defquasiconvex} A set $A\subset X$ is \emph{quasiconvex} if there exists a constant $C$ such that any two points $a,b\in A$ can be joined by a curve that lies in $A$ and has length at most $C d_X(a,b)$. 
\end{definition}
 
We write $\mathcal H^s$ for the $s$-dimensional Hausdorff measure~\cite[p. 55-56]{Mattila}, that is
\begin{equation*}
{\mathcal H^s}(A)=\lim_{\delta\downarrow0} {\mathcal H^s_{\delta}}(A),
\end{equation*}
where
\begin{equation*}
{\mathcal H^s_{\delta}}(A)=\inf \left\{\sum_{i=1}^\infty\diam(E_i)^s:\ A\subset \bigcup_{i=1}^\infty E_i,\ \diam(E_i) \le \delta\right\}.
\end{equation*}

\section{Basic properties of metrically removable sets}

\begin{lemma}\label{equivalentdef} A subset $E$ of a metric space $X$ is metrically removable if and only if it has empty interior and $\rho_{E^c}(a,b) = d(a,b)$ for all $a,b\in E^c$. 
 \end{lemma}

\begin{proof} If $E$ is metrically removable, then any two points $a,b\in X$ are connected by a curve that is contained in $E^c$, except possibly for its endpoints. Therefore, $E^c$ is dense in $X$, which means $E$ has empty interior. The equality  $\rho_{E^c} = d$ is immediate. 

Conversely, suppose $E$ has empty interior and $\rho_{E^c} = d$. Given $a,b\in X$ and $\epsilon>0$, pick two sequences $\{a_k\}$ and  $\{b_k\}$ in $E^c$ such that $d(a_k, a)<\epsilon/2^k$ and $d(b_k,b)<\epsilon/2^k$ for all $k\in\mathbb N$. Note that 
\begin{equation}\label{equiv1}
d(a_1,b_1) < d(a,b)+\epsilon,\quad  d(a_k, a_{k+1})< \frac{\epsilon}{2^{k-1}}, \quad  d(b_k, b_{k+1})< \frac{\epsilon}{2^{k-1}}.
\end{equation}
Let $\gamma_0\subset E^c$ be a curve from $a_1$ to $b_1$ such that $\ell(\gamma_0)\le (1+\epsilon)d(a_1, b_1)$. For every $k$, there is a path $\gamma_k\subset E^c $ from $a_k$ to $a_{k+1}$ with $\ell(\gamma_k)\le (1+\epsilon) d(a_k, a_{k+1})$. Similarly, there is a path $\gamma_k'\subset E^c $ from $b_k$ to $b_{k+1}$ with $\ell(\gamma_k')\le (1+\epsilon) d(b_k, b_{k+1})$.  

Concatenating all the curves $\gamma_k$ and $\gamma_k'$, and adding $a,b$ as the endpoints, we obtain a continuous  curve that connects $a$ to $b$ and is disjoint from $E\setminus \{a,b\}$. Its length is bounded from above by
\[
(1+\epsilon)\left(d(a_1, b_1) + \sum_{n=1}^\infty (d(a_k, a_{k+1})+d(b_k, b_{k+1})) \right)
\]
which according to~\eqref{equiv1} is at most $(1+\epsilon)(d(a,b) + 5\epsilon)$, proving that $E$ is metrically removable.  
\end{proof}

Metrically removable sets can be seen as ``thin'' in several ways.

\begin{lemma}\label{totallydisconnected} A metrically removable set $E\subset \mathbb R^2$ is totally disconnected. 
\end{lemma}

\begin{proof} Pick any point $a\in E$, without loss of generality $a=0$. Since $E$ has empty interior by Lemma~\ref{equivalentdef}, there exist four points $b_1,\dots,b_4\in E^c$, such that each $b_k$ lies in the $k$th open quadrant of the plane and $|b_k|<\epsilon$. Connecting these points by line segments $[b_1,b_2], \dots, [b_4,b_1]$ we get a closed polygonal curve $\gamma$ with $0$ in its interior domain. Let $d=\dist(0, \gamma)$ and replace each segment of $\gamma$ by a curve that is contained in $E^c$ and is short enough to stay in the $(d/2)$-neighborhood of the segment. The resulting closed curve separates $0$ from the circle $|z|=2\epsilon$. Since  $\epsilon$ was arbitrarily small, the lemma is proved.
\end{proof}

Since a line in $\mathbb R^n$ is metrically removable for $n\ge 3$, the statement of  Lemma~\ref{totallydisconnected} does not extend to higher dimensions. 

Any metrically removable set has quasiconvex complement, while the converse is false: for example, a ball in $\mathbb R^n$, $n\ge 2$, has quasiconvex complement but is not metrically removable. However, for closed sets of zero area these notions coincide. 

\begin{proposition}\label{quasiconvex2removable} Suppose that $E\subset \mathbb R^n$  is a closed set such that $\mathcal H^n(E)=0$ and $E^c$ is quasiconvex. Then $E$ is metrically removable. 
\end{proposition}

\begin{proof} Fix distinct points $a,b\in E^c$ and pick $\epsilon>0$ small enough so that $B(a,\epsilon)$ and $B(b,\epsilon)$ are disjoint from $E$.  By Fubini's theorem, almost every line parallel to $[a,b]$ intersects $E$ along a set of zero length. Thus we can choose $a', b'\in E^c$ such that $|a-a'|<\epsilon$, $|b-b'|<\epsilon$, and $\mathcal H^1(E\cap [a',b'])=0$. 

Since $E\cap [a',b']$ is a compact set of zero length, it can be covered by finitely many disjoint open intervals $(p_k, q_k)$ of total length less than $\epsilon$. For each $k$ there is a curve $\gamma_k\subset E^c$ that joins $p_k$ to $q_k$ and has length at most $C|p_k-q_k|$, where $C$ is the constant of quasiconvexity of $E^c$. Removing $[p_k,q_k]$ from $[a',b']$ and inserting   $\gamma_k$  instead, we obtain a curve $\gamma$ that joins $a'$ to $b'$ and has length less than $|a'-b'|+C\epsilon$. Then $[a,a']\cup \gamma\cup [b',b]$ is a curve of length at most 
\[
|a'-b'|+(C+2)\epsilon \le |a-b|+(C+4)\epsilon
\]
which shows $\rho_{E^c}(a,b) = |a-b|$.  Since $E$ has empty interior, Lemma~\ref{equivalentdef} implies it is metrically removable.
\end{proof}

\begin{corollary}\label{powerRemovable} If $A\subset \mathbb R$ is a closed set and $\mathcal H^1(A) = 0$, then $A^n$ is metrically removable in $\mathbb R^n$ for all $n\ge 2$.
\end{corollary}

\begin{proof} By Theorem~A~\cite{HakobyanHerron}, the set $\mathbb R^n\setminus A^n$ is quasiconvex whenever $A$ is a closed subset of $\mathbb R$ with empty interior, and $n\ge 2$. It remains to apply Proposition~\ref{quasiconvex2removable}. 
\end{proof}

For example, the product of two standard middle-third Cantor sets $C$ is metrically removable in $\mathbb R^2$ by Corollary~\ref{powerRemovable}. This shows that metric removability cannot be characterized in terms of Hausdorff dimension:  we have $\dim(C\times C)= \log 4/\log 3 > 1$, while a line segment is not metrically removable in $\mathbb R^2$.  An even more extreme example is given below.

\begin{proposition}\label{positiveMeasure} For $n\ge 2$ there exist metrically removable compact sets $E\subset \mathbb R^n$ with $\mathcal H^n(E)>0$.
\end{proposition}

\begin{proof} Let $A\subset \mathbb R^n$ be the union of all line segments with  endpoints in $\mathbb Q^n$. Since $\mathcal H^n(A)=0$, the complement $A^c$ contains a compact set $E$ of positive $\mathcal H^n$ measure. 

To show that $E$ is metrically removable, fix distinct points $a,b\in E^c$ and $\epsilon>0$ where $\epsilon<\dist(E, \{a,b\})$. 
There are points $a'\in \mathbb Q^n \cap B(a,\epsilon)$ and $b'\in \mathbb Q^n\cap B(b,\epsilon)$. The polygonal curve $[a,a']\cup [a',b']\cup [b',b]$ is disjoint from $E$ and has length less than $|a-b|+4\epsilon$. By Lemma~\ref{equivalentdef}, the set $E$ is metrically removable.
\end{proof}

Hakobyan and Herron~\cite{HakobyanHerron} constructed totally disconnected compact sets in $\mathbb  R^n$ with non-quasiconvex complement. Their sets have a prescribed Hausdorff dimension in $[n-1,n]$. 
As a consequence, there is a rich supply of totally disconnected compact sets which are not metrically removable in $\mathbb R^n$. 

In Proposition~\ref{quasiconvex2removable}, the assumption that the set has zero measure is essential. The following proposition provides examples of sets with quasiconvex complement which are not metrically removable, even though some of them are totally disconnected. 

\begin{proposition}\label{fatCantorproduct} If $A\subset \mathbb R$ is a set of positive Lebesgue measure, then the product $A\times A$ is not metrically removable in $\mathbb R^2$. \end{proposition}

\begin{proof} Since $A$ contains a compact subset of positive measure, we may assume $A$ itself is compact. By the Lebesgue density theorem, there exists an interval $I$ such that $\mathcal H^1(A\cap I) > 0.9 \mathcal H^1(I)$. We may assume $I=(0,1)$ and $A\subset I$ without loss of generality. 

Let $\gamma$ be a curve that connects $(0,0)$ to $(1,1)$ and is disjoint from $A\times A$. Since the distance from $A\times A$ to $\gamma$ is positive, we may and do replace $A$ by a larger subset of $(0,1)$ that consists of finitely many closed intervals, so that $\gamma$ is still disjoint from $A\times A$.   

Let $m = \mathcal H^1(A)$. Define the function $f\colon \mathbb R \to \mathbb R$ by $f(x) = 0$ for $x\le 0$ and $f(x) = \mathcal H^1([0,x]\cap A)$ for $x>0$. This is a $1$-Lipschitz function that maps $\mathbb R$ onto $[0, m]$. Therefore, the map $F(x,y) = (f(x), f(y))$ is also $1$-Lipschitz and its range is the square $Q = [0,m]\times [0,m]$. 

The set $F(E^c)$ consists of the boundary of $Q$ and finitely many horizontal and vertical segments connecting  the opposite sides of $Q$. The set $F(\gamma)$ connects opposite corners of $Q$ and is contained in $F(E^c)$. Therefore, the length of $F(\gamma)$ is at least twice the sidelength of $Q$. Recalling that $F$ is $1$-Lipschitz, we conclude that 
\[
\ell(\gamma) \ge 
\mathcal H^1(F(\gamma)) \ge 2\mathcal H^1(A) > 1.8. 
\]
Since the distance between the endpoints of $\gamma$ is $\sqrt{2}<1.8$, the set $A\times A$ is not metrically removable. 
\end{proof}

The property of having quasiconvex complement is not inherited by subsets: for example, a disk in $\mathbb R^2$ has quasiconvex complement but a line segment does not. On the other hand, Definition~\ref{defmetricremovable} makes it clear that any subset of a metrically removable set is metrically removable.

\begin{lemma}\label{Omegaremovable} If $\Omega$ is a domain in $\mathbb R^n$ and $E\subset\mathbb R^n$ is a metrically removable set, then  $\rho_{\Omega\setminus E}$ agrees with  $\rho_\Omega$ on $\Omega\setminus E$. 
\end{lemma}

\begin{proof} Given $a,b\in\Omega\setminus E$ and $\epsilon>0$, let $\gamma$ be a polygonal curve which connects $a$ to $b$ within $\Omega$ and has length less than $|a-b| + \epsilon / 2$. We may assume that the vertices of $\gamma$ are in $E^c$, since $E$ has empty interior.

Let $L_1,\dots, L_N$ be the line segments of the polygonal curve $\gamma$. Also let $d=\dist(\gamma, \Omega^c)$. For $k=1,\dots, N$ replace $L_k$ by a curve $\Gamma_k$ that connects the endpoints of $L_k$ within $E^c$ and satisfies 
$\ell(\Gamma_k) < \ell(L_k) + \delta$ where $\delta<\epsilon/(2N)$ and is small enough to ensure that $\Gamma_k$ stays in the open $d$-neighborhood of $L_k$. The concatenation of $\Gamma_k$ is a curve of total length less than $|a-b|+\epsilon$ which connects $a$ to $b$ within $\Omega\setminus E$. 
\end{proof}

\begin{lemma}\label{countableunion} The countable union of metrically removable closed sets in $\mathbb R^n$ is metrically removable.
\end{lemma}

\begin{proof} Suppose $E=\bigcup_{k=1}^\infty E_k$ where each $E_k$ is closed and metrically removable in $\mathbb R^n$. Since each $E_k$ has empty interior, their union $E$ is a set of first category and therefore also has empty interior. By virtue of Lemma~\ref{equivalentdef}, it remains to show that $\rho_{E^c}(a,b) = |a-b|$ for $a,b\in E^c$. 

Fix $\epsilon>0$. There is a polygonal curve $\gamma_1$ of length less than $|a-b|+\epsilon/2$ which connects $a$ to $b$ in $E_1^c$. We may assume that the vertices of $\gamma_1$ lie in $E^c$ since they can be moved slightly to avoid $E$. 

Once a curve $\gamma_k$ has been constructed, we construct $\gamma_{k+1}$ as follows. Let $N_k$ be the number of segments in $\gamma_k$, and let $d_k=\dist(\gamma_k, \bigcup_{j\le k} E_j)$.
Also define $\delta_k = 2^{-k-1} \min_{j\le k}d_j$. 
Since $E_{k+1}$ is metrically removable, we can replace each line segment $L$ of $\gamma_k$ with a polygonal curve that has  vertices in $E^c$, is disjoint from $E_{k+1}$, has length less than 
$\ell(L) < 2^{-k-1}\epsilon/N$, and is contained in the $\delta_k$-neighborhood of $L$ (the latter is made possible by Lemma~\ref{Omegaremovable}). 

The resulting curve $\gamma_{k+1}$ has length less than $|a-b| + \epsilon$. Consider  its constant-speed parameterization with $[0,1]$ as the domain. By the equicontinuity of these parameterizations, the sequence $\gamma_k$ has a subsequence that converges uniformly to some curve $\gamma$ of length at most $|a-b|+\epsilon$. 

It remains to check that $\gamma$ is disjoint from $E$. To this end it suffices to show that $\dist(\gamma, E_k) >0$ for all $k$.
By construction, for $m \ge k$  the curve $\gamma_{m+1}$ is contained in the $\delta_m$-neighborhood of $\gamma_m$, where $\delta_m \le d_k / 2^{m+1} $. Therefore, $\gamma$ is contained in the $(d_k/2)$-neighborhood of $\gamma_k$. This implies 
$\dist(\gamma, E_k) \ge d_k / 2 > 0$, completing the proof. 
\end{proof}

In Lemma~\ref{countableunion} it is essential that the sets are assumed closed (although their union need not be). For example, both $[0,1]\cap \mathbb Q$ and $[0,1]\setminus \mathbb Q$ are metrically removable in $\mathbb C$, but their union is not.

\section{Estimates for the intrinsic metric}\label{secIntrinsic}

The main tool for proving Theorem~\ref{metricallyremovable} is the following lemma of independent interest. 

\begin{lemma}\label{lengthestimate} 
For any domain $\Omega\subset \mathbb C$ we have
\begin{equation}\label{lengthestimateeq}
\rho_\Omega(a,b)\le |a-b| + \frac{\pi}{2}\,\mathcal H^1(\partial\Omega)
\end{equation}
for all $a,b\in\Omega$.
\end{lemma}

The proof of Lemma~\ref{lengthestimate} involves the concept of Painlev\'e length from~\cite[p. 48]{Garnett}.

\begin{definition}\label{defPainleve} The \emph{Painlev\'e length} of a compact set $K\subset \mathbb C$, denoted $\kappa(K)$, is the infimum of numbers $\ell$ with the following property: for every open set $U$ containing $K$ there exists an open set $V$ such that $K\subset V\subset U$ and $\partial V$ is a finite  union of disjoint analytic Jordan curves of total length at most $\ell$. 
\end{definition}

Instead of analytic curves, one could use smooth or merely rectifiable curves in Definition~\ref{defPainleve} without changing the value of $\kappa(K)$. Indeed, if $\gamma$ is a  rectifiable Jordan curve, let $\Phi$ be a conformal map of the exterior of the unit disk onto the exterior domain bounded by $\gamma$. The images of circles $|z|=r$ under $\Phi$ are  analytic Jordan curves, and their length converges to the length of $\gamma$ as $r \to 1^+$. 

\begin{proposition}\label{kappaH} \cite[p. 25]{Dudziak} The inequality $\kappa(K) \le \pi \mathcal H^1(K)$ holds  for every compact set $K\subset\mathbb C$. 
\end{proposition}

\begin{proof}[Proof of Lemma~\ref{lengthestimate}] 
Let $L=|a-b| + \frac{\pi}{2}\,\mathcal H^1(\partial\Omega)$ and 
$K=\partial\Omega\cap \overline{B(a,L)}$. It suffices to work with $K^c$ instead of $\Omega$, because a path from $a$ to $b$ of length sufficiently close to $L$ cannot exit $B(a,L)$. Note also that $\mathcal H^1(K)\le \mathcal H^1(\partial\Omega)$.  

Fix $\epsilon>0$. Since $K^c$ contains $\Omega$, there is a curve $\Gamma$ connecting $a$ and $b$ in $K^c$. Pick an open set $U$ such that $K\subset U$ and $\dist(U, \Gamma)>0$. By Proposition~\ref{kappaH} there exists an open set $V$ such that $K\subset V\subset U$ and $\partial V$ is a finite disjoint union of   analytic Jordan curves $\sigma_j$, $j=1,\dots,N$, of total length at most $\pi \mathcal H^1(K) + \epsilon$. By construction,  each $\sigma_j$ is disjoint from $K$. Also, $a$ and $b$ are in the same connected component of $\overline{V}^c$, being connected by the curve $\Gamma$. 

Let $\gamma_0(t) = (1-t)a + tb$ be the line segment $[a,b]$ parameterized by $t\in [0,1]$. If $\gamma_0$ does not meet $\partial V$, then it is contained in $K^c$ and we are done. Otherwise, let $t_1 = \min \{t\colon \gamma_0(t)\in \partial V\}$. The point $\gamma_0(t_1)$ belongs to some Jordan curve $\sigma_j$. If $\sigma_j$ has no other intersection point with $\gamma_0$, then it separates $a$ from $b$, which is impossible. Let $t_2 = \max \{t\colon \gamma_0(t) \in \sigma_j\}$. The line segment $\gamma_0([t_1,t_2])$ can be replaced by the shorter of two subarcs of the Jordan curve $\sigma_j$ determined by the points 
$\gamma_0(t_1)$ and $\gamma_0(t_2)$. This adds at most $\ell(\sigma_j)/2$ to the length. 

The remaining part $\gamma_0([t_2,1])$ no longer meets $\sigma_j$.  Therefore, repeating the above process will result, in finitely many steps, in a curve $\gamma$ connecting $a$ to $b$ within $K^c$. This curve consists of parts of the segment $[a,b]$ and arcs of the curves $\sigma_j$, and satisfies 
\[
\ell(\gamma) \le |a-b|+\sum_{j=1}^N \frac{\ell(\sigma_j)}{2} 
\le |a-b|+ \frac{\pi}{2}\, \mathcal H^1(K) + \frac{\epsilon}{2}. 
\]
This proves~\eqref{lengthestimateeq}. 
\end{proof}

\begin{proof}[Proof of Theorem~\ref{metricallyremovable}] We consider the case $n=2$ first. Let $E\subset \mathbb C$ be a closed totally disconnected set with $\mathcal H^1(E) = L <\infty$. Fix distinct points $a,b\in E^c$ and pick $\epsilon>0$ small enough so that $B(a,\epsilon)$ and $B(b,\epsilon)$ are disjoint from $E$.  Since the length of $E$ is finite, almost every line parallel to $[a,b]$ has finite intersection with $E$~\cite[Theorem~10.10]{Mattila}. Choose $a', b'\in E^c$ such that $|a-a'|<\epsilon$, $|b-b'|<\epsilon$, and $E\cap [a',b'] = \{z_1,\dots, z_N\}$ is finite. 

Choose $r>0$ small enough so that 
\begin{itemize}
\item $r < \epsilon/N$;
\item $|z_k-z_j|> 2r$ whenever $k\ne j$; 
\item $\mathcal H^1(E\cap B(z_k, r)) < \epsilon/N$ for each $k=1,\dots,N$. 
\end{itemize}
By Lemma~\ref{lengthestimate} for each $k$ there exists a curve $\gamma_k\subset B(z_k, r) \setminus E$ which joins two points of $[a',b']\cap B(z_k, r)$ separated by $z_k$ and has length at most 
\[
2r + \frac{\pi}{2} \mathcal H^1(\partial (B(z_k, r)\setminus E))
\le 2r + \pi^2 r + \frac{\pi \epsilon }{2 N}.
\]
Using each $\gamma_k$ as a detour around $z_k$, we obtain a curve that joins $a'$ to $b'$ and has length at most 
\[
|a' - b'| + 2rN + \pi ^2 r N + \frac{\pi \epsilon }{2}
< |a-b| + (4+\pi^2 + \pi/2)\epsilon 
\]
which proves the theorem since $[a,a']$ and $[b,b']$ are disjoint from $E$. 

Now suppose $n\ge 3$. Given $a,b\in E^c$ and $\epsilon>0$, fix a two-dimensional plane $P$ containing $a$ and $b$. By~\cite[Theorem~10.10]{Mattila}, the intersection $E\cap (P+v)$ has finite length for almost every vector $v$ orthogonal to $P$. Since $E$ is closed, we can choose such $v$ with $|v|<\min(\epsilon, \dist(\{a,b\}, E))$. Applying the two-dimensional case to $E\cap (P+v)$, we obtain a curve $\gamma$ that joins $a+v$ to $b+v$ within $(P+v)\setminus E$ and has length less than $|a-b|+\epsilon$. The concatenation of $\gamma$ with the segments $[a,a+v]$ and $[b,b+v]$ joins $a$ to $b$ in $E^c$ and has length less than $|a-b|+3\epsilon$. 
\end{proof}

Unlike Theorem~\ref{metricallyremovable}, Lemma~\ref{lengthestimate} does not extend to higher dimensions: when $n\ge 3$, there is no universal constant $C$ such that every domain $\Omega\subset\mathbb R^n$ satisfies 
\[
\rho_\Omega(a,b)\le C(|a-b|+\mathcal H^{n-1}(\partial \Omega))
\quad \text{for all} \ a,b \in \Omega. 
\]
Indeed, we can connect two points $a,b\in \mathbb R^n$ by a very long circular arc and let $\Omega$ be a small tubular neighborhood of that arc; then $\mathcal H^{n-1}(\partial \Omega)$ is small.

\section{Removable sets for functions with bounded derivative}\label{secBoundedDer}

Carleson~\cite{Carleson} proved that sets of zero area are removable for Lipschitz functions, and the converse was proved later by Uy~\cite{Uy}. A Lipschitz-continuous holomorphic function has bounded derivative; however, the converse is in general false. The following proposition shows that the class of removable sets for functions with bounded derivative is much smaller than for Lipschitz functions.

\begin{proposition} A connected compact set with more than one point is not removable for holomorphic functions with bounded derivative.
\end{proposition}
\begin{proof} Let $K$ be such a set. There is a conformal map $f\colon \overline{\mathbb C}\setminus  K \to \mathbb D$ such that $f(\infty)= 0$. The square of $f$ is $O(1/|z|^2)$ as $z\to\infty$ and therefore has zero residue at infinity. This makes its antiderivative $F(z) = \int^z f(\zeta)^2\,d\zeta$ a holomorphic function in $K^c$. Clearly, $|F'|=|f^2|<1$ in $K^c$. If $F$ could be extended to an entire function, $F'$ would be a bounded entire function and therefore constant. This is impossible since $F'(z)\to 0$ as $z\to\infty$. 
\end{proof}

However, in a quasiconvex domain the boundedness of derivative implies Lipschitz continuity, since one can integrate the derivative along paths of controlled length. Therefore, every compact set of zero area with quasiconvex complement is removable for functions with bounded derivative. This leads to the following corollary of Theorem~\ref{metricallyremovable}.  

\begin{corollary}\label{H1removable} If $K\subset \mathbb C$ is a totally disconnected compact set and $\mathcal H^1(K) < \infty$, then $K$ is removable for holomorphic functions with bounded derivative.
\end{corollary}

It is clear that a set of zero analytic capacity is removable for functions with bounded derivative, since its complement does not support any nonconstant bounded holomorphic functions. However, Corollary~\ref{H1removable} also applies to some sets of positive capacity, such as a totally disconnected compact subset of $\mathbb R$ with positive length. 

\section{Comparison of thinness conditions}\label{secThinness}

Tabor and Tabor~\cite{TaborTabor} introduced the concept of ``intervally thin'' sets, which is related to removability of sets for convex functions~\cite{PokornyRmoutil, TaborTabor}.  

\begin{definition}\label{defintervallythin}~\cite{TaborTabor} A set $E\subset \mathbb R^n$ is \emph{intervally thin} if for all $a,b\in \mathbb R^n$ and $\epsilon>0$ there exist points $a',b'$ such that $|a-a'|<\epsilon$, $|b-b'|<\epsilon$, and the line segment $[a',b']$ is disjoint from $E$. 
\end{definition}

This concept is closely related to metric removability: the reader may wish to observe that the set constructed in  Proposition~\ref{positiveMeasure} is intervally thin. 
Definition~\ref{defintervallythin} can be rephrased as: any two open balls in $\mathbb R^n$ can be connected by a line segment disjoint from $E$. The latter statement is made more precise by the following result. 

\begin{lemma}\label{improveintervallythin} Suppose $E\subset \mathbb R^n$ is intervally thin. Let $P$ and $Q$ be distinct $(n-1)$-dimensional hyperplanes in $\mathbb R^n$. Then for any two points $p\in P$, $q\in Q$ and any $r>0$ the sets $A = P\cap B(p, r)$ and $B = Q\cap B(q,r)$ can be connected by a line segment disjoint from $E$.
\end{lemma}

\begin{proof} Since both $A$ and $B$ are $(n-1)$-dimensional disks not contained in the same hyperplane, the difference set $A-B=\{a-b\colon a\in A, b\in B\}$ is $n$-dimensional.  Therefore, there exist $a\in A$ and $b\in B$ such that the vector $a-b$ is not parallel to either $P$ or $Q$. Let $L$ be the line through $a$ and $b$. Pick two points $a_1,b_1\in L$ such that both $a$ and $b$ lie strictly between $a_1$ and $b_1$.  

For sufficiently small $\epsilon>0$ any line segment connecting $B(a_1, \epsilon)$ to $B(b_1, \epsilon)$ 
intersects both $A$ and $B$. Since $E$ is intervally thin, some of such line segments are disjoint from $E$, proving the claim. 
\end{proof}

In order to obtain a sufficient removability condition for  holomorphic functions with restricted argument of derivative (Theorem~\ref{argremovable}), we need the concept of a Lipschitz-thin set, which is developed in the remainder of this section.

\begin{definition}\label{lipgraph} Let $\epsilon>0$. 
A curve $\gamma\colon [\alpha, \beta]\to \mathbb R^n$ is an $\epsilon$-Lipschitz graph if for every $\alpha \le t<s\le \beta$ the angle between the vectors $\gamma(s)-\gamma(t)$ and $\gamma(\beta)-\gamma(\alpha)$ is less than $\epsilon$.  
\end{definition}

\begin{definition}\label{defLipthin} A set $E\subset \mathbb R^n$ is \emph{Lipschitz-thin} if for any $\epsilon>0$, any two points  $a,b\in \mathbb R^n$ can be connected by an $\epsilon$-Lipschitz graph that is disjoint from $E\setminus \{a,b\}$.   
\end{definition}

The following result is a counterpart of Lemma~\ref{equivalentdef} for Lipschitz-thin sets.

\begin{lemma}\label{equivalentLipthin} A set $E\subset \mathbb R^n$ is Lipschitz-thin if and only if it has empty interior and any two points  $a,b\in E^c$ can be connected by an $\epsilon$-Lipschitz graph within $E^c$.
\end{lemma}

The proof of Lemma~\ref{equivalentLipthin} relies on a geometric fact which we isolate into a lemma. 

\begin{lemma}\label{doublesequenceofballs} For any distinct points $a,b\in \mathbb R^n$ and any $\epsilon>0$ there exists a double-infinite sequence $\{x_k\colon k\in\mathbb Z\}\subset [a,b]$ and positive numbers $r_k>0$ such that 
\begin{enumerate}[(a)]
\item $x_k\to a$ as $k\to-\infty$ and $x_k\to b$ as $k\to\infty$
\item For any choice of points $y_k\in B_k:= B(x_k, r_k)$, the angle between the vectors $y_{k}-y_{k-1}$ and $b-a$ is less than $\epsilon$. 
\end{enumerate}
\end{lemma}
\begin{proof} Without loss of generality we may assume $a=0$ and $|b|=1$. Let $\delta>0$ be a small number to be chosen later. Define 
\[
x_k = \begin{cases} \delta^{-k}b,\  &k <0; \\
(1-\delta^{k+1})b,\ &k\ge 0,
\end{cases}
\quad \text{and} \quad  
r_k = \begin{cases} \delta^{-2k},\   &k <0; \\
\delta^{2k+2} ,\   &k\ge 0.
\end{cases}
\]
Observe that 
\begin{equation}\label{balls1}
|x_{k}-x_{k-1}| = 
\begin{cases} \delta^{|k|}-\delta^{|k|+1}, \ & k \ne 0; \\ 
1-2\delta, \ & k = 0. 
\end{cases}
\end{equation}
For any choice of points $y_k\in B_k$ we have 
\[
|y_k-y_{k-1}|\le |x_{k}-x_{k-1}|  + r_k + r_{k-1}. 
\]
On the other hand, writing $P$ for the orthogonal projection onto the line along $b$, we have 
\[
|P(y_k)-P(y_{k-1})|\ge |x_{k}-x_{k-1}|  - r_k - r_{k-1}. 
\]
Comparing~\eqref{balls1} with the definition of $r_k$, we find that $(r_k + r_{k-1})/|x_{k}-x_{k-1}| \le C\delta$ with $C$ independent of $k$ or $\delta$. By choosing $\delta$ sufficiently small, we can make the ratio 
\[
\frac{|P(y_k)-P(y_{k-1})|}{|y_k-y_{k-1}|}
\ge \frac{|x_{k}-x_{k-1}|  - r_k - r_{k-1}}{|x_{k}-x_{k-1}|  + r_k + r_{k-1}} \ge \frac{1-C\delta}{1+C\delta}
\]
arbitrarily close to $1$, which implies the conclusion of the lemma. 
\end{proof}

\begin{proof}[Proof of Lemma~\ref{equivalentLipthin}] The necessity part is clear. To prove sufficiency, fix $a,b\in \mathbb R^n$ and $\epsilon>0$, and let $B_k$ be as in Lemma~\ref{doublesequenceofballs}. For each $k\in\mathbb Z$ pick $y_k\in B_k\setminus E$ which is possible because $E$ has empty interior. 

Connect each $y_k$ to $y_{k-1}$ by an $\epsilon$-Lipschitz graph $\gamma_k\subset E^c$. The concatenation of these curves is a curve from $a$ to $b$ that lies in $E^c$ except possibly its endpoints. By construction, this curve is a $(2\epsilon)$-Lipschitz graph.
\end{proof}

As another application of Lemma~\ref{doublesequenceofballs}, we relate the notions of ``intervally thin" and ``Lipschitz-thin'' on the plane. 

\begin{proposition}\label{interval2Lip} Any intervally thin set   $E\subset \mathbb R^2$ is Lipschitz thin. 
\end{proposition}

\begin{proof} An intervally thin set has empty interior by definition. Fix $\epsilon>0$ and distinct points $a,b\in E^c$. We may assume $a=0$ and $b=1$, identifying $\mathbb R^2$ with $\mathbb C$. Let $B_k$, $k\in\mathbb Z$, be the disks provided by Lemma~\ref{doublesequenceofballs}. 

Since $E$ is intervally thin, there exists a line segment 
$L_1\subset E^c$ connecting $B_0$ to $B_1$. By Lemma~\ref{improveintervallythin}, there is a line segment $L_2$ connecting $L_1\cap B_1$ to $B_2$. Continuing in this way, let $L_{k+1}\subset E^c$ be a line segment connecting $L_k\cap B_k$ to $B_{k+1}$. After erasing extraneous parts of segments $L_k$ (namely, the part of $L_k\cap B_k$ extending beyond the point $L_k\cap L_{k+1}$) we obtain a curve that begins with $L_1$ and ends at $1$. 

Similarly, let $L_{0}\subset E^c$ be a line segment connecting $L_1\cap B_0$ to $B_{-1}$ and erase the part of $L_1\cap B_0$ extending beyond $L_1\cap L_0$, etc. This process results in the curve
\[
\gamma = \{0,1\}\cup \bigcup_{k\in\mathbb Z} L_k
\]
which is the desired $\epsilon$-Lipschitz graph connecting $0$ to $1$ within $E^c$.
\end{proof}

The proof of Proposition~\ref{interval2Lip} breaks down in  dimensions $n>2$, where Lemma~\ref{improveintervallythin} provides a way to connect $(n-1)$-dimensional disks instead of $1$-dimensional line segments. However, we still have such a result for closed sets. 

\begin{proposition}\label{closedLipthin} If a closed set $E\subset \mathbb R^n$ is intervally thin, then it is Lipschitz-thin. 
\end{proposition}

\begin{proof} Given distinct points $a,b\in E^c$ and $\delta\in (0,1)$, pick $r>0$ be such that both $B(a,r)$ and $B(b,r)$ are disjoint from $E$ and $r<\delta|a-b|$. Let $u=(b-a)/|b-a|$ and define
\[
a_1 = a + (1-\delta)r u,\quad \text{and}\quad 
b_1 = b - (1-\delta)r u.
\]
The balls $B(a_1, r\delta)$ and $B(b_1, r\delta)$ are connected by some line segment $[a',b']\subset E^c$. The piecewise linear curve $aa'b'b$ is disjoint from $E$, and all three of its segments are nearly parallel to vector $u$ when $\delta$ is small enough. Thus, $aa'b'b$ is the desired $\epsilon$-Lipschitz graph with small $\epsilon$.  
\end{proof}

The converse of Proposition~\ref{interval2Lip} is false: a Lipschitz-thin set need not be intervally thin, as the following two examples show.

\begin{example}\label{Lipnotinterval} In $\mathbb R^2$, let $E=I_1\cup I_2\cup I_3$ where 
\[
\begin{split}
I_1&=\{(x,1)\colon x\in [-1,1]\setminus \mathbb Q\}; \\  
I_2&=\{(x,0)\colon x\in [-1,1]\cap \mathbb Q\}; \\  
I_3&=\{(x,-1)\colon x\in [-1,1]\setminus \mathbb Q\}. 
\end{split}
\]
The set $E$ is Lipschitz-thin, because all three sets $I_k$ are easily avoided by a polygonal path that can be made arbitrarily close to straight. On the other hand, there is no line segment that connects a small neighborhood of $(0, 2)$ to a small neighborhood of $(0, -2)$ while avoiding $E$. Indeed, such a line segment $L$ would contain two points $(u,1)$ and $(v, -1)$ with $u,v\in \mathbb Q$. Then $((u+v)/2, 0)\in L\cap E$, proving the claim.
\end{example}

The set $E$ in Example \ref{Lipnotinterval} is not closed. A compact set with the same properties can be constructed with the following iterative process.   

\begin{example}\label{Lipnotintervalcompact}
Let $\delta = 2^{-5}$.
Define for every $n,k \in \mathbb{N}$ and $i \in \{0,1,2\}$ a similitude mapping $f_{n,i,k}\colon \mathbb R^2 \to \mathbb R^2$ by setting
\[
f_{n,i,k}(x) = \begin{cases}
(1-2\delta)2^{-2n}x + (i\cdot 2^{-n}, (k+\delta)\cdot 2^{-2n}), & \text{if }i=0, 1\\
(1-2\delta)2^{-2n}x + (i\cdot 2^{-n}, (k+\frac12+\delta)\cdot 2^{-2n}), & \text{if }i = 2.
\end{cases}
\]
Given a set $F\subset \mathbb R^2$, let 
\[
S_n(F) = \bigcup_{i=0}^2 \bigcup_{k=0}^{2^{2n}-2} f_{n,i,k}(F)
\]
and define a sequence of compact sets $E_7 \supset E_8 \supset \cdots$ as
\[
E_n = S_7\circ S_8 \circ \cdots \circ S_n([0,1]^2).
\]
The final compact set is defined as
\[
 E = \bigcap_{n=7}^\infty E_n.
\]
Let us then show that the set $E$ is Lipschitz-thin, but not intervally thin.

In order to see that $E$ is not intervally thin we prove  
\begin{quote}
\textbf{Claim 1}:
any line passing through  $\{0\} \times [2\delta,1-2\delta]$
 making an angle at most $\pi/4$ with the horizontal axis, must intersect the set $E$.
 \end{quote}
Assuming Claim 1, any line segment connecting a point in $B((-1/2,1/2),1/4)$ to a point in $B((3/2,1/2),1/4)$ intersects $E$ and thus $E$ is not intervally thin.

To prove that $E$ satisfies Claim~1, it is enough to show that for any set $F \subset [0,1]^2$ satisfying Claim~1 and for any $n \ge 7$ also the set $S_n(F)$
 satisfies Claim~1. Indeed, assuming this is true, then since $[0,1]^2$ satisfies Claim~1, so does each $E_n$, and the property carries over to the nested intersection $E$.
 
 For showing that $S_n(F)$ satisfies Claim 1, let $L$ be a line passing through $\{0\} \times [2\delta,1-2\delta]$ and
 making an angle at most $\pi/4$ with the horizontal axis. Suppose towards a contradiction that $S_n(F) \cap L = \emptyset$.
 Denote for $i \in \{0,1,2\}$ by $y_i \in [\delta, 1-\delta]$ the $y$-coordinate of the intersection of $L$ with the vertical line $\{i\cdot 2^{-n}\} \times \mathbb R$. By the facts that $n \ge 7$, $L$ passes through $\{0\} \times [2\delta,1-2\delta]$ and
 makes an angle of at most $\pi/4$ with the horizontal axis, we have that $L$ intersects $\{2^{-n}\} \times [\delta,1-\delta]$ and $\{2^{-n+1}\} \times [\delta,1-\delta]$.
This together with the assumption $S_n(F) \cap L = \emptyset$ and the definition of $S_n$ implies that there exist $k_0,k_1, k_2 \in \mathbb N$ such that
 \begin{equation}\label{eqfirst2}
  |y_i - k_i2^{-2n}| \le 3\delta\cdot 2^{-2n}, \qquad \text{for }i = 0,1
 \end{equation}
 and
 \begin{equation}\label{eqlast}
 \left|y_2 - \left(k_2 + \frac12\right)2^{-2n}\right| \le 3\delta \cdot 2^{-2n}.
 \end{equation}
 Since $L$ is a line, $y_2 =  2y_1 - y_0 $. Combining this with~\eqref{eqfirst2} gives the estimate
 \begin{align*}
  |y_2 - (2k_1 - k_0)2^{-2n}| & = |2y_1-y_0 - (2k_1 - k_0)2^{-2n}|\\
  & \le 2|y_0 - k_02^{-2n}| + |y_1 - k_12^{-2n}|\\
  & \le 9\delta\cdot 2^{-2n}.
 \end{align*}
 Since $\delta = 2^{-5}$, this contradicts \eqref{eqlast}. Thus $S_n(F) \cap L \ne \emptyset$ and Claim 1 holds.
 
 It remains to show that $E$ is Lipschitz-thin. Fix $\epsilon>0$. Observe that 
 \[
 S_n([0,1]^2) \subset [0,2^{-n+2}]\times[0,1]
 \]
which implies that $S_{n-1}\circ S_n([0,1]^2)$ is contained in vertical strips of width $2^{-3n+4}$ separated by horizontal distances at least $2^{-n}$. Furthermore, each vertical strip has holes of height at least $\delta\cdot 2^{-2n+2}$ placed uniformly at vertical distance less than $2^{-2n+2}$ from one another. These holes allow curves to pass through the vertical strips with only a slight change of direction. Therefore, there exists $n\in\mathbb N$ such that any two points $x,y \in \mathbb R^2$ with distance at least $\epsilon$ from $S_{n-1}\circ S_n([0,1]^2)$ can be connected by an $\epsilon$-Lipschitz graph avoiding $S_{n-1}\circ S_n([0,1]^2)$. 

Subsequent application of $S_{n-2},\dots, S_7$ only replicates the above at smaller scales, since the property of being an $\epsilon$-Lipschitz graph is preserved under similitudes. Therefore, any two points $x,y \in \mathbb R^2$ with distance at least $\epsilon$ from $E_n$ can be connected by an $\epsilon$-Lipschitz graph avoiding $E_n$. Consequently, $E$ is Lipschitz-thin. \qed
\end{example}

It is immediate that a Lipschitz-thin set is metrically negligible. The converse is not true, as the following example, called ``Holey Devil's Staircase'' in~\cite{PokornyRmoutil}, shows.

\begin{example} Let $C\subset [0,1]$ be the standard middle-third Cantor set, and let $f\colon [0,1]\to [0,1]$ be the associated ``staircase'' function, i.e., the continuous function that is constant on each component of $[0,1]\setminus C$, where it is equal to the midpoint of the component.  Let $E=\{(x,f(x)) \colon x\in C\}$ be the part of the graph of $f$ that lies over $C$. This is a compact totally disconnected set which is metrically removable but not Lipschitz-thin.
\end{example}

\begin{proof} Since $E$ is a subset of the graph of an increasing function, its $\mathcal H^1$ measure is finite. By  Theorem~\ref{metricallyremovable} $E$ is metrically removable.

Suppose that $g\colon [0,1]\to [0,1]$ is a strictly increasing function such that $g(0)>0$ and $g(1)<1$. We claim that the graph of $g$ meets $E$. Indeed, let $x_0=\inf\{x\colon g(x)=f(x)\}$; this infimum is defined because $g(0) > f(0)$ and $g(1) < f(1)$.  If $x_0\notin C$ then consider $x_1<x_0$ such that $f(x_1)=f(x_0)$. Since $g(x_1)<g(x_0)=f(x_0)=f(x_1)$, the intermediate value theorem implies that $g=f$ at some point of $(0, x_1)$, contradicting the choice of $x_0$.  

It remains to observe that for $\epsilon< \tan^{-1}(1/3)$, any $\epsilon$-Lipschitz graph connecting the points $(0, 1/3)$ and $(1, 2/3)$ is the graph of a strictly increasing function $g$ to which the previous paragraph applies.  
\end{proof}

\section{Extension of delta-monotone maps}\label{secExtension}

The extension theorem of this section will be applied to holomorphic functions in \S\ref{secargremovable}. 

\begin{definition}\label{defDM}\cite{Kovalev} Let $\Omega\subset \mathbb R^n$, $n\ge 2$, and $\delta>0$. A map $f\colon \Omega\to\mathbb R^n$ is called \emph{$\delta$-monotone} if 
\begin{equation}\label{defDM1}
\langle f(x)-f(y), x-y \rangle  \ge \delta  |f(x)-f(y)| |x-y|
\end{equation}
for all $x,y\in \Omega$.  
\end{definition}

Examples of $\delta$-monotone maps are easy to find when $n=2$, by taking $\Omega$ to be a convex domain and $f$ a holomorphic function such that $|\arg f'|\le \cos^{-1}\delta$. 
For example~\cite[Example 15]{Kovalev}, the function $f(z)=-1/z$ is $\delta$-monotone  in the domain $\Omega = \{z\in\mathbb C\colon |\arg z|<\pi/6\}$, with $\delta=1/2$. Observe that this function does not have a continuous extension to $\overline{\Omega}$, being unbounded near $0$. The following theorem shows this is the only obstruction to continuous extension. 

\begin{definition}\label{defconnectedatboundary} An open set $\Omega\subset\mathbb R^n$ is \emph{locally connected on the boundary} if for every $b\in\partial \Omega$ and every $r>0$ there exists an open set $U$ such that $b\in U\subset B(b,r)$ and $U\cap \Omega$ is connected. 
\end{definition}

\begin{theorem}\label{extendDM} Suppose $\Omega\subset \mathbb R^n$, $n\ge 2$,  is open and locally connected on the boundary. Let $f\colon \Omega\to\mathbb R^n$ be a $\delta$-monotone map that is bounded on bounded subsets of $\Omega$. 
Then $f$ has a continuous extension to $\overline{\Omega}$, which is also $\delta$-monotone.
\end{theorem}

We need additional notation for the proof. Given a point $p\in \mathbb R^n$, a nonzero vector $v\in \mathbb R^n$, and an angle $\theta\in (0,\pi/2)$, let
\[
C(p,v,\theta) = \{x\in \mathbb R^n\colon 
\langle x-p, v\rangle \ge \cos\theta |x-p||v|\}
\]
be the closed cone with vertex $p$, the axis parallel to $v$, and opening angle $\theta$. Note that if $f\colon \Omega\to\mathbb R^n$ is a $\delta$-monotone map and $\alpha:=\theta+\cos^{-1}\delta <\pi/2$, then 
\begin{equation}\label{imageofcone}
f(C(p,v,\theta)\cap \Omega) \subset C(f(p), v, \alpha)
\end{equation}
for any $p\in \Omega$. 

Let us say that $p$ is a \textit{vertex} of a set $E\subset\mathbb R^n$ if $p\in E$ and there exist $v\ne 0$ and $\theta\in (0,\pi/2)$ such that  $E\subset C(p,v,\theta)$.  

\begin{lemma}\label{nonvertex} For any set $E\subset \mathbb R^n$ the set of vertices of $E$ is countable. 
\end{lemma}
\begin{proof} Suppose $p\in E\subset C(p,v,\theta)$. Then for every vector $y$ in the interior of the dual cone $C(0, v, \pi/2-\theta)$ the linear function $x\mapsto \langle x, y\rangle$  attains its minimum on $E$ at the point $p$ and nowhere else. Therefore, the dual cones associated with distinct vertices of $E$ are disjoint. Since there can be only countably many disjoint open subsets of $\mathbb R^n$, the lemma is proved.   
\end{proof}

\begin{proof}[Proof of Theorem~\ref{extendDM}]
Fix $b\in \partial \Omega$. For each $k\in \mathbb N$ let $U_k$ be an open subset of $\mathbb R^n$ such that $b\in U_k\subset B(b,1/k)$ and $U_k\cap \Omega$ is connected. Define 
\[
E = \bigcap_{k=1}^\infty E_k,\quad\text{where } \ E_k = \overline{f(U_k\cap\Omega)}.  
\]
Then each set $E_k$ is nonempty, compact and connected. The intersection of a nested sequence of such sets is nonempty, compact, and connected as well~\cite[Theorem~28.2]{Willard}.

Choose $\alpha$ strictly between $\cos^{-1}\delta$ and $\pi/2$. 
Fix $y\in E$ and pick a sequence $x_j\to b$ such that $f(x_j)\to y$. Passing to a subsequence, we may assume 
\begin{equation}\label{directionslimit}
\frac{b-x_j}{|b-x_j|}\to u\quad \text{as } n\to\infty,
\end{equation}
where $u$ is some unit vector. For a fixed $j$, we have 
\[
B(b,r) \subset C(x_j, b-x_j, \alpha-\cos^{-1}\delta)
\]
when $r>0$ is small enough. By ~\eqref{imageofcone} this implies
$f(B(b,r)) \subset  C(f(x_j), b-x_j, \alpha)$, hence $E\subset C(f(x_j), b-x_j, \alpha)$. 

Passing to the limit $j\to\infty$ and using~\eqref{directionslimit}, we obtain $E\subset C(y, u, \alpha)$. Thus, every point of $E$ is a vertex. 
By Lemma~\ref{nonvertex} the set $E$ is countable. Being also nonempty and connected, $E$ must consist of precisely one point, say $E=\{y\}$. This implies $\lim_{x\to b}f(x) =  y$, which provides the desired continuous extension of $f$ to the boundary. 
Finally, the extended map is $\delta$-monotone because the inequality~\eqref{defDM1} is preserved under taking limits.
\end{proof}

\begin{corollary} \label{extendDMdense} 
Suppose $\Omega\subset \mathbb R^n$, $n\ge 2$, is open, dense in $\mathbb R^n$, and locally connected on the boundary. Then every $\delta$-monotone map $f\colon \Omega\to\mathbb R^n$  has a continuous $\delta$-monotone extension to $\mathbb R^n$. 
\end{corollary}

\begin{proof} In view of Theorem~\ref{extendDM} we only need to prove that $f(B(0,r)\cap \Omega)$ is bounded for every $r>0$. Choose $\alpha$ strictly between $\cos^{-1}\delta$ and $\pi/2$. When $R$ is sufficiently large, we have 
\[
B(0,r)\subset C(x, -x, \alpha-\cos^{-1}\delta)
\]
for all $x\in\Omega$ with $|x|\ge R$. Since $\Omega$ is open and dense, there is $x\in \Omega$ such that $-x\in\Omega$ and $|x|\ge R$. From~\eqref{imageofcone} it follows that 
\[
f(B(0,r)\cap\Omega) \subset C(f(x), -x, \alpha) \cap C(f(-x), x, \alpha) 
\]
where the set on the right is bounded, proving the claim. 
\end{proof}

The relevance of the $\delta$-monotonicity condition to the extension theorem~\ref{extendDM} is emphasized by the following example. 

\begin{example} Let $\Omega = \mathbb R^n\setminus \{0\}$ where $n\ge 2$. Define $f\colon \Omega\to\mathbb R^n$ by $f(x) = x + x/|x|$. An interested reader can check that $\langle f(x)-f(y), x-y\rangle >0$ for all pairs of distinct points $x,y\in\Omega$. Yet, $f$ does not have a continuous extension to $0$. It narrowly fails the $\delta$-monotonicity condition~\eqref{defDM1}. 
\end{example}

\begin{remark}\label{domainassumption} Every quasiconvex domain $\Omega$ is locally connected on 
the boundary. In particular, when $E\subset \mathbb R^n$ is closed and metrically removable, its complement $\Omega = E^c$ satisfies the assumptions of Corollary~\ref{extendDMdense}.
\end{remark}

Indeed, given $b\in \partial \Omega$ and $r>0$,  one can use quasiconvexity to find sufficiently small $\rho<r/2$ so that any two points of $B(b,\rho)\cap \Omega$ are connected by a curve of length less than $r/2$. Such a curve must lie within $B(b,r)$. Therefore, $B(b,\rho)\cap \Omega$ belongs to one connected component of $B(b,r)\cap \Omega$, which satisfies Definition~\ref{defconnectedatboundary}.  

\section{Removable sets for functions with restricted argument of derivative}\label{secargremovable}

\begin{theorem}\label{argremovable} Let $K\subset \mathbb C$ be a closed Lipschitz-thin set with $\mathcal H^2(K)=0$. Suppose $f\colon K^c\to \mathbb C$ is holomorphic and there exists $\alpha<\pi/2$ such that
\begin{equation}\label{argrestriction}
|\arg f'(z)| \le \alpha,\quad z\in K^c
\end{equation}
(in particular, $f'\ne 0$.) Then $f$ extends to an entire function, which is in fact linear.
\end{theorem}

The first step toward the proof of Theorem~\ref{argremovable}, presented as a lemma below, does not rely on $K$ having zero measure. 

\begin{lemma}\label{thin2DM} Let $K\subset \mathbb C$ be a closed Lipschitz-thin set. Suppose $f\colon K^c\to \mathbb C$ is holomorphic and satisfies~\eqref{argrestriction} with $\alpha<\pi/2$. Then $f$ is $\delta$-monotone with $\delta=\cos \alpha$.
\end{lemma}

\begin{proof} Fix distinct $z,w\in K^c$. Pick $\epsilon < \pi/2-\alpha $ and let $\gamma$ be an $\epsilon$-Lipschitz graph connecting $w$ to $z$ within $K^c$. When parameterized by its arclength, $\gamma$ satisfies
\[
\left|\arg \frac{\gamma'(t)}{z-w}\right| \le \epsilon 
\]
for almost all $t$ in its parameter interval. Using the inequality $|\arg f'|<\alpha$ we obtain 
\[
\left|\arg \frac{(f\circ \gamma)'}{z-w}\right| < \alpha+\epsilon. 
\]
Since $f\circ \gamma $ is absolutely continuous, integration yields 
\[
\left|\arg \frac{f(z)-f(w)}{z-w}\right| < \alpha+\epsilon 
\]
which implies~\eqref{defDM1} with $\delta = \cos(\alpha+\epsilon)$. Since $\epsilon$ can be arbitrarily small, the lemma is proved.  
\end{proof}

\begin{proof}[Proof of Theorem~\ref{argremovable}] By Lemma~\ref{thin2DM}, the map $f$ is $\delta$-monotone with $\delta=\cos \alpha$. Corollary~\ref{extendDMdense} with Remark~\ref{domainassumption} provide its $\delta$-monotone extension $F$ to the entire complex plane. 

A $\delta$-monotone map $F\colon \mathbb C\to\mathbb C$ is quasiconformal~\cite[Theorem~6]{Kovalev}, which means that $F$ is locally in the Sobolev space $W^{1,2}$ and satisfies the Beltrami equation 
\[
\frac{\partial F}{\partial \bar z} = \mu(z)\frac{\partial F}{\partial z}
\]
almost everywhere in $\mathbb C$, with $\mu$ being a measurable complex-valued function such that $\esssup|\mu|<1$. 

Since $F$ is holomorphic on $K^c$, its Beltrami coefficient $\mu$ is zero a.e. The uniqueness theorem for the Beltrami equation (\cite[Theorem V.B.1]{Ahlfors} or~\cite[Theorem 5.3.4]{AIMbook}) implies that such $F$ must be a linear function, as claimed.
\end{proof}

\section{Remarks and questions}\label{secQuestions}

A homeomorphism of $\mathbb R^2$ does not preserve metric removability of sets in general. Indeed, there exists a homeomorphism $g\colon \mathbb R\to\mathbb R$ that maps the standard Cantor set $C$ onto a Cantor-type set $C'$ of positive measure. Let $E=C\times C$ and $f(x,y)=(g(x), g(y))$. Then $f\colon \mathbb R^2\to \mathbb R^2$ is a homeomorphism, the set $E$ is metrically removable by Corollary~\ref{powerRemovable} while $f(E) = C'\times C'$ is not metrically removable by Proposition~\ref{fatCantorproduct}.  Note that the map $f$ in this example is neither Lipschitz nor quasiconformal.  

\begin{question}\label{BLinvariance} Are metrically removable sets preserved by bi-Lipschitz homeomorphisms $f\colon \mathbb R^n\to\mathbb R^n$? Or even by quasiconformal maps? 
\end{question}

The property of having quasiconvex complement is obviously preserved by bi-Lipschitz maps. So, the class of closed metrically removable sets $E\subset \mathbb R^n$ with $\mathcal H^n(E)=0$ is indeed preserved by bi-Lipschitz homeomorphisms, by virtue of Lemma~\ref{quasiconvex2removable}.   

\begin{question} What is the best constant in \eqref{lengthestimateeq}? It seems likely that $\pi/2$ can be improved. The example of $\Omega=\mathbb C\setminus [-1,1]$ with $a,b=\pm \epsilon i$ shows that the constant should be at least $1$. Is the inequality 
\begin{equation}\label{lengthestimateconj}
\rho_\Omega(a,b)\le |a-b| + \mathcal H^1(\partial\Omega)
\end{equation}
true? 
\end{question}

As is observed in \cite[p.~26]{Dudziak}, Proposition~\ref{kappaH} holds in the stronger form $\kappa(K)\le \pi \mathcal H^1_\infty(K)$, that is, with the Hausdorff measure $\mathcal H^1$ is replaced by the Hausdorff content $\mathcal H^1_\infty$. In the Hausdorff content version, the constant $\pi$ cannot be improved because for the unit disk $\overline{\mathbb D}$ we have $\kappa(\overline{\mathbb D})=2\pi$ and $\mathcal H^1_\infty(\overline{\mathbb D}) = \diam \overline{\mathbb D} = 2$. However, we do not know of such an example for Hausdorff measure. 

\begin{question} Can the constant $\pi$ in Proposition~\ref{kappaH} be improved? The best constant cannot be less than $3$ because a modification of Sierpinski gasket described in~\cite[p. 75]{Mattila} has $\mathcal H^1(K)=1$ and $\kappa(K) = 3$. 
\end{question}

\begin{question} Is every intervally thin set Lipschitz-thin? By the results of \S\ref{secThinness} this is true in two dimensions, and for closed sets in all dimensions. 
\end{question}

\bibliographystyle{amsplain} 
\bibliography{references.bib}

\end{document}